\documentclass[12pt]{amsart}
\usepackage[utf8]{inputenc}

\usepackage[english]{babel}
\usepackage[T1]{fontenc}
\usepackage{todonotes}
\usepackage{bm}

\usepackage{amsmath}
\usepackage{amsthm}
\usepackage{enumitem}
\usepackage{amssymb,amsfonts,stmaryrd,mathrsfs,mathabx}
\usepackage{amsmath,latexsym,amsthm,mathtools,bbm}
\usepackage{pgf,tikz} 
\usepackage{subcaption}

\usepackage[colorlinks,cite color=green,link color=purple]{hyperref}
\usepackage{todonotes}
\usepackage[thinlines]{easytable}
\usepackage{thmtools}
\usepackage[capitalize]{cleveref}

\newtheorem{thm}{Theorem}[section]
\newtheorem{cor}[thm]{Corollary}
\newtheorem{lem}[thm]{Lemma}
\newtheorem{prop}[thm]{Proposition}
\newtheorem{defi}[thm]{Definition}
\newtheorem{conj}{Conjecture}

\theoremstyle{remark}
\newtheorem{remark}{Remark}
\newtheorem{example}{Example}[section]



\newcommand{\bfp}{\boldsymbol{p}}
\newcommand{\bfx}{\boldsymbol{x}}
\newcommand{\bfy}{\boldsymbol{y}}

\newcommand{\bfq}{\boldsymbol{q}}
\newcommand{\bfr}{\boldsymbol{r}}

\newcommand{\bfu}{\boldsymbol{u}}

\newcommand{\Jch}{\theta^{(\alpha)}}

\newcommand{\tH}{\widetilde{H}}
\newcommand{\J}{{J}^{(\alpha)}}

\newcommand{\G}{\mathcal G}
\newcommand{\tG}{\widetilde{\mathcal G}}
\newcommand{\Op}{\mathcal{O}}

\newcommand{\B}{\mathcal B}

\newcommand{\C}{\mathcal C}

\newcommand{\ZZ}{\mathbb Z}
\newcommand{\QQ}{\mathbb Q}
\newcommand{\NN}{\mathbb N}

\newcommand{\mcA}{\mathcal{A}}
\newcommand{\supernabla}{{\nabla}_{\bfy}}
\newcommand{\nablaqt}{\nabla}
\newcommand{\dehsupernabla}{\widetilde{\nabla}_{\boldsymbol{y}}}
\newcommand{\Pop}{\mathcal{P}}
\newcommand{\Top}{\mathcal{T}}
\newcommand{\Topqt}{\mathcal{T}}

\DeclareMathOperator{\Exp}{Exp}
\DeclareMathOperator{\End}{End}
\DeclareMathOperator{\ad}{ad}
\DeclareMathOperator{\Span}{\text{Span}}
\newcommand{\bA}{\mathbb A}
\newcommand{\bB}{\mathbb B}
\newcommand{\bD}{\mathbb D}

\newcommand{\YY}{\mathcal{P}}

\newcommand{\mcO}{\mathcal{O}}
\newcommand{\mcJ}{\mathcal{J}}

\newcommand{\mcP}{\mathcal{P}}

\newcommand{\mcS}{\mathcal{S}}
\newcommand{\mcSh}{\widehat{\mathcal{S}}}
\newcommand{\mcSstar}{\mathcal{S}^*}
\newcommand{\N}{\mathcal{N}}
\newcommand{\tN}{{\mathcal{\widetilde N}}}

\usepackage[colorlinks,cite color=green,link color=purple]{hyperref}
\definecolor{green}{RGB}{43,92,47}
\definecolor{blue}{RGB}{40,68,104}
\definecolor{red}{RGB}{254, 113, 96}
\definecolor{purple}{RGB}{102,0,51}
\definecolor{gray}{RGB}{224,224,224}
\definecolor{lightpurple}{RGB}{255, 249, 242}
\definecolor{blue}{RGB}{40,68,104}
\hypersetup{colorlinks = true, urlcolor=blue,linkbordercolor=red}
\let\underbrace\LaTeXunderbrace

\author[Houcine Ben Dali]{Houcine Ben Dali}
\address{\parbox{\linewidth}{Department of Mathematics, Harvard University, Cambridge, MA 02138, USA and 
Center for Mathematical Sciences and Applications, Harvard University, Cambridge, MA 02138, USA}}

\email{bendali@math.harvard.edu}

\title{A formula for the Jack super nabla operator}

\begin{document}

\begin{abstract}
    We study a Jack analog $\nabla(\bfp,\bfq)$ of the super nabla operator recently introduced by Bergeron, Haglund, Iraci and Romero for Macdonald polynomials. We prove that $\nabla(\bfp,\bfq)$ has a differential expression in the power-sum basis given in terms of Chapuy--Do\l{}e\k{}ga and Nazarov--Sklyanin operators.

This result is obtained from a more general formula for the operator $\G(\bfp,\bfq)$ encoding the structure coefficients of Jack characters, from which $\nabla(\bfp,\bfq)$ is obtained by taking the top homogeneous part.

A key step of the proof involves establishing that Chapuy--Do\l{}e\k{}ga operators together with a dehomogenized version of Nazarov--Sklyanin operators have a Heisenberg algebra structure.
The proof also uses a characterization of the operator $\G(\bfp,\bfq)$ with a family of differential equations, recently established by the author. 
\end{abstract}

\maketitle

\section{Introduction}
\subsection{Jack polynomials and the super nabla operator}
Bergeron, Haglund, Iraci and Romero introduced in \cite{BergeronHaglundIraciRomero2023} a new operator called the super nabla operator, defined by its diagonal action on modified Macdonald polynomials. This operator is a natural generalization of many well-studied operators in the theory, including the classical nabla operator and the Delta operator. The authors establish several combinatorial formulas for the specialization at $t=1$ of this operator, and formulate positivity conjectures for general $q$ and $t$.

In this paper we study a Jack analog of the super nabla operator. 
Jack polynomials $\J_\lambda$ are symmetric functions introduced by Jack \cite{Jack1970}. They are indexed by integer partitions and depend on a deformation parameter $\alpha$. Jack polynomials are obtained from modified Macdonald polynomials by applying a plethystic substitution and taking a suitable limit.  They have a rich combinatorial structure \cite{Stanley1989,KnopSahi1997,ChapuyDolega2022,Moll2023} and are related to various models of statistical and quantum mechanics \cite{LapointeVinet1995,DumitriuEdelman2002,For}. 

We denote $\J_\lambda(\bfp)$ and $\J_\lambda(\bfq)$ the Jack polynomial indexed by the partition $\lambda$, which depend respectively on two different alphabets of power-sum variables $\bfp:=(p_1,p_2,\dots)$ and $\bfq:=(q_1,q_2,\dots)$. We define the \emph{Jack super nabla operator} by 
$$\nabla(\bfp,\bfq)\cdot \J_\lambda(\bfp)=\J_\lambda(\bfp)\J_\lambda(\bfq).$$

Note that the operator $\nabla(\bfp,\bfq)$ is homogeneous: when applied to a homogeneous function of degree $n$ in 
$\bfp$, it produces a function that is homogeneous of degree $n$
in both $\bfp$ and $\bfq$.
We define the following \textit{dehomogenization} of the super nabla operator 
\begin{equation}\label{eq:nabla_G}
  \G(\bfp,\bfq)=\exp\left(\frac{\partial}{\partial p_1}\right)\exp\left(\frac{\partial}{\partial q_1}\right)
\nabla(\bfp,\bfq)
\exp\left(-\frac{\partial}{\partial p_1}\right).  
\end{equation}
The operator $\G(\bfp,\bfq)$ can alternatively be defined by a diagonal action on the series of Jack characters (see \cref{prop:G_mcJ}) or by its action on power-sum symmetric functions, encoded by the structure coefficients of Jack characters (see \cref{eq:def_g}).

It follows from the definitions that $\nabla(\bfp,\bfq)$ is obtained as the top homogeneous part of $\G(\bfp,\bfq)$: if $f$ is homogeneous of degree $n$, then
$$\nabla(\bfp,\bfq)\cdot f(\bfp)=\left[\G(\bfp,\bfq)\right]^{(0,n)}\cdot f(\bfp),$$
where $\left[\mcO(\bfp,\bfq)\right]^{(i,j)}$ is the homogeneous part of the operator $\mcO(\bfp,\bfq)$ of degree $i$ in $\bfp$ and $j$ in $\bfq$.

\subsection{Chapuy--Do\l{}e\k{}ga operators and Nazarov--Sklyanin operators}\label{ssec:diff_operators}
Chapuy and Do\l{}e\k{}ga introduced  a family of differential operators\footnote{The operators introduced in \cite{ChapuyDolega2022} are denoted $\B_n$; we use here a variant  that appeared in \cite{BenDaliDolega2023}.} $\{\C_\ell\}_{\ell\geq 0}$ which they used to solve the two alphabet-case in Goulden--Jackson's \emph{$b$-conjecture}. These operators also played a central role in establishing a new combinatorial formula for Jack polynomials in terms of \emph{maps} \cite{BenDaliDolega2023}, allowing to prove a conjecture of Lassalle about Jack characters. 

Chapuy--Do\l{}e\k{}ga operators have a rich algebraic and combinatorial structure: they can be defined by nested commutators using the Laplace--Beltrami operator, they encode the combinatorics of bipartite maps, counted with non-orientability weights, and they can be described using lattice paths. 
In this paper, we use the definition of $\{\C_\ell(\bfp,u)\}_{\ell\geq 0}$ using differential expressions, which involve an additional alphabet of ``catalytic variables'' (see \cref{ssec:CD}). As an example, we give here the terms of degree 1 and 2 in $\C_0$ and $\C_1$:

\begin{multline*}
\C_0=-\sum_{i\geq 1}p_{i+1}\frac{i \partial}{\partial p_i}+\frac{1}{2}\left(\sum_{i\geq 1}(i+1)(\alpha-1) p_{i+2}\frac{i \partial}{\partial p_i}+\sum_{i\geq 1}\sum_{\substack{j+k=i+2\\j,k\geq 1}}p_{j}p_k\frac{i \partial}{\partial p_i}\right.\\
\left.+\alpha\sum_{i,j\geq 1}p_{i+j+2}\frac{i\partial}{\partial p_i} \frac{j\partial}{\partial p_j} \right)
  +\text{terms of higher degree},
\end{multline*}

\begin{align*}
\C_1=-\frac{p_1}{\alpha}+\left(\sum_{i\geq 1}p_{i+2}\frac{i\partial}{\partial p_i}+\frac{(\alpha-1)}{2\alpha}p_2+\frac{p_1^2}{2\alpha}\right)+\text{terms of higher degree}.
\end{align*}
In general, the operator $\C_\ell$ has terms of degree $k$, for $k\geq \ell+\delta_{0,\ell}$.

On the other hand, Nazarov and Sklyanin introduced in \cite{NazarovSklyanin2013b,NazarovSklyanin2013} a family of homogeneous operators $\{\N_\ell\}_{\ell\geq 0}$, obtained as a limit of Sekiguchi–Debiard operators (see \cref{ssec:NS} for the definition). These operators are closely connected to the quantum Benjamin--Ono integrable system. They commute pairwise and act diagonally on Jack polynomials, they also admit elegant combinatorial interpretations in terms of lattice paths \cite{Moll2015,CuencaDolegaMoll2023}.  In this paper, we introduce a family of inhomogeneous operators $\{\C_{-\ell}\}_{\ell\geq 1}$, obtained from Nazarov--Sklyanin operators as follows
\begin{equation}\label{eq:def_C_neg}
   C_{-\ell}:=(-1)^\ell[v^{\ell+1}]\log\left(\sum_{\ell\geq 0}\exp\left(\frac{\partial}{\partial p_1}\right) v^\ell\N_\ell (\bfp) \exp\left(-\frac{\partial}{\partial p_1}\right)\right),  
\end{equation}
where the quantity inside the logarithm is considered as a formal power-series in the variable $v$.
For example, we have
$$-\C_{-1}=\frac{\alpha \partial}{\partial p_1}+\sum_{i\geq 1}p_i\alpha\frac{i\partial}{\partial p_i}$$
\begin{multline*}
  \C_{-2}=\alpha\frac{2\partial}{\partial p_2}+\left((\alpha-1)\frac{\alpha\partial}{\partial p_1}+2\sum_{i\geq 1}p_i\alpha\frac{(i+1)\partial}{\partial p_{i+1}}\right)\\
  +\left(\sum_{i,j\geq 1}p_ip_j\alpha\frac{(i+j)\partial}{\partial p_{i+j}}+\sum_{i,j\geq 1}p_{i+j}\alpha^2\frac{i\partial}{\partial p_i}\frac{j\partial}{\partial p_j}+\sum_{i\geq 1}\alpha(\alpha-1)p_i\frac{i\partial}{\partial p_i}\right).  
\end{multline*}
In general, the operator $\C_{-\ell}$ has terms of degree $k$, for $-\ell\leq k\leq 0.$

\begin{remark}
    The operator $\N_\ell$ can be interpreted as the generating series of lattice paths with $\ell$ steps of the form $(1,k)$ for $k\in \ZZ$, starting at the origin, ending at $(\ell,0)$, and remaining weakly above the horizontal axis. In this generating function, an up step (resp. a down step) of size $k$ acts as $p_k$ (resp. $\alpha\frac{ k \partial}{\partial p_k}$), and a horizontal step at height $i$ acts as $(\alpha-1)i$. We refer to \cite[Section 3]{CuencaDolegaMoll2023} for more details. Similarly, using a combinatorial argument one can show that the operator $\C_{-\ell}$ can be described as the generating function of paths with $\ell+1$ steps and such that:
    \begin{itemize}
        \item each path is counted with an extra factor $\frac{1}{r}$, where $r+1$ is the number of times the path touches the horizontal axis. This is related to the logarithm in \cref{eq:def_C_neg}.
        \item some of the up steps of the path are forgotten: they act as one instead of $p_1$. This is related to the conjugation by $\exp(\frac{\partial}{\partial p_1})$ in \cref{eq:def_C_neg}.
    \end{itemize}
\end{remark}
\subsection{The Heisenberg algebra of the  operators \texorpdfstring{$\C_\ell$}{Cl}.}
The first main contribution of this paper is to prove that the dehomogenized Nazarov--Sklyanin operators and Chapuy--Do\l{}e\k{}ga operators satisfy a family of unexpected commutation relations giving them the structure of a Heisenberg algebra.
\begin{thm}[First main result]\label{thm:commutation_relations}
    We have the following relations between the operators $\left\{\C_\ell\right\}_{\ell\in \ZZ}$:
    $$[\C_m,\C_\ell]=\delta_{\ell,-m}\ \textup{sign}(\ell), \quad\text{for $\ell,m\neq  0$},$$
    where 
    $$\textup{sign}(\ell)=\begin{cases}
        +1 & \text{if $\ell>0$,}\\
        -1 & \text{if $\ell<0$.}
    \end{cases}$$
    Moreover,
    \begin{equation}\label{eq:commutation_relations_C0}
      [\C_0,\C_\ell]=
      \begin{cases}
      (\ell+1)\C_{\ell+1}, &\quad\text{if $\ell\geq  -1$},  \\
      \ell\C_{\ell+1}     & \quad\text{if $\ell\leq  -2$}.    
      \end{cases}
    \end{equation}
\end{thm}
As a consequence, we get a representation of the Heisenberg algebra given by the action of the operators $\C_\ell$, different from the classical representation, where it simply acts by multiplication and derivatives with respect to power-sums. We then have the following correspondence between the two representations:
$$\begin{array}{ccc}
    p_\ell
    &\longleftrightarrow &\C_\ell\\
   \frac{\partial}{\partial p_\ell}
   &\longleftrightarrow &\C_{-\ell}\\
   \sum_{i\geq 1}p_{i+1}\frac{(i+1)\partial}{\partial p_i}
   &\longleftrightarrow &\C_0.
\end{array}$$
It is worth noticing that while in the first representation the Heisenberg algebra acts on the space of symmetric functions, in the second one it acts on the space of formal power-series; see \cref{ssec:SymFun} for more details.

In the following, we briefly describe the different parts of the proof of \cref{thm:commutation_relations}:
\begin{itemize}
    \item The case $\ell,m\geq0$ has been proved in \cite[Theorem 6.6]{BenDaliDolega2023}. The proof is technically involved and uses computations with extra catalytic variables.
    \item The case $\ell,m\leq0$, is a consequence of the fact that Nazarov--Sklyanin operators are diagonal on Jack polynomials; see \cref{thm:comm_rel_neg}.
    \item The case $\ell\geq 0$ and $m\leq0$ uses the connection of the two families of operators with Jack characters, as well as plethystic-type computations; see \cref{thm:commutation_relations_pos_neg}.
\end{itemize}

Given the combinatorial nature of the operators $\C_\ell$, it is natural to ask whether there exists a combinatorial proof of the commutation relations of \cref{thm:commutation_relations}. We leave this question as an open problem.

\subsection{A differential expression for the super nabla operator}
The second main result of this paper is a differential expression of the operator $\G$ in terms of the operators $\{C_\ell\}_{\ell\in \ZZ}$.

\begin{thm}[Second main result]\label{thm:main_thm}
    We have
    \begin{equation}\label{eq:G_diff}
      \G(\bfp,\bfq)=\exp\left(\sum_{\ell\geq 1}C_\ell(\bfq)C_{-\ell}(\bfp)\right).  
    \end{equation}
    In particular, for any homogeneous symmetric function $f$ of degree $n$, we have
    $$\nabla(\bfp,\bfq)\cdot f(\bfp)=\left[\exp\left(\sum_{\ell\geq 1}C_\ell(\bfq)C_{-\ell}(\bfp)\right)\right]^{(0,n)}\cdot f(\bfp),$$
    where $[.]^{(0,n)}$ denotes the homogeneous part of the operator which has degree $0$ in $\bfp$ and degree $n$ in $\bfq$.
\end{thm}

The proof of this theorem is based on \cref{thm:commutation_relations} and on a characterization of the operator $\G$ with commutation relations established in \cite{BenDali2025} (see \cref{thm:characterization_G}).

\subsection{Connection to the Matching-Jack conjecture and other open problems}
\subsubsection{The Matching-Jack conjecture}
Goulden and Jackson introduced in \cite{GouldenJackson1996} a family of coefficients $c^\pi_{\mu,\nu}(\alpha)$ indexed by three partitions of the same size, defined by the expansion of a multiparametric Jack series in the power-sum basis, see \cref{eq:tau_c}. It turns out that these coefficients can alternatively be defined using the super nabla operator:
\begin{equation}\label{eq:def_c}
  \nabla(\bfp,\bfq) p_\pi=\sum_{\mu,\nu\vdash n}c^\pi_{\mu,\nu}(\alpha)p_\mu q_\nu.  
\end{equation}
We prove in \cref{prop:tau_c} the equivalence between the two formulations.  

\begin{conj}[\cite{GouldenJackson1996}]\label{conj:GJ}
    The coefficients $c^\pi_{\mu,\nu}(\alpha)$ are in $\NN[b]$, where $b$ is the shifted Jack parameter $b:=\alpha-1$.
\end{conj}
This conjecture remains open despite many partial results \cite{KanunnikovVassilieva2016,DolegaFeray2016,BenDali2022,BenDali2023}.
The reformulation of the Matching-Jack conjecture given by \cref{eq:def_c}, and which mainly follows from the orthogonality of Jack polynomials, seems to be new. It makes a connection between this long-standing problem in the Jack theory, with several well-studied problems in the Macdonald literature; see \cite{CarlssonMellit2018,DAdderioMellit2022,BlasiakHaimanMorsePunSeelinger2023b,BergeronHaglundIraciRomero2023}.

We consider now the action of the operator $\G(\bfp,\bfq)$ on the power-sum basis. It defines a family of coefficients $g^\pi_{\mu,\nu}$ (for partitions  $(\pi,\mu,\nu)$ not necessarily of the same size):
\begin{equation}\label{eq:def_g}
  \G(\bfp,\bfq) p_\pi=\sum_{\mu,\nu}g^\pi_{\mu,\nu}(\alpha)p_\mu q_\nu.  
\end{equation}
When $|\pi|=|\mu|=|\nu|$, we have $c^{\pi}_{\mu,\nu}=g^{\pi}_{\mu,\nu}$. In \cref{prop:g}, we prove that the coefficients $g^{\pi}_{\mu,\nu}$ correspond to the \emph{structure coefficients} of Jack characters, which are the object of a generalized version of \cref{conj:GJ} due to Śniady \cite[Conjecture 2.2]{Sniady2019}.

\subsubsection{Other positivity conjectures}
The reformulation of the Matching–Jack conjecture using the super nabla operator naturally raises the question of whether this operator is positive on other bases. We formulate the following conjectures, tested\footnote{We provide \textbf{\href{https://github.com/HoucineBenDali/Jack-super-nabla-operator}{here}} a Sage code allowing to compute the coefficients appearing in each one of these conjectures.} for $n\leq 9$.

\begin{conj}\label{conj:monomials}
Fix $n\geq 1$. For three partitions $\pi,\mu,\nu\vdash n$,
define the coefficient $d^\pi_{\mu,\nu}(\alpha)$ by
    the expansion
    $$\nabla(\bfp,\bfq) p_\pi=\sum_{\mu,\nu\vdash n}d^\pi_{\mu,\nu}p_\mu m_\nu(\bfq),$$
    where $m_\nu(\bfq)$ denotes the monomial symmetric functions in the alphabet $\bfq$.
    Then $d^\pi_{\mu,\nu}(\alpha)\in\NN[\alpha].$ 
\end{conj}
Note that \cref{conj:GJ} and the fact that power-sums are monomial positive imply that $d^\pi_{\mu,\nu}\in\NN[b]$. \cref{conj:monomials} suggests that this positivity holds for $\alpha$, without applying the shift.

\begin{remark}
When $\alpha=1$, it is known that the coefficients $c^\pi_{\mu,\nu}(1)$ count some families of orientable bipartite maps with conditions on the vertex and face degrees depending on the partitions $(\pi,\mu,\nu)$ (see e.g. \cite{GouldenJackson1996}). It can then be shown that $d^\pi_{\mu,\nu}(1)$ count some orientable bipartite maps with \textit{colored faces}. \cref{conj:monomials} is then equivalent to finding a statistic $\eta$ on these maps such that $d^\pi_{\mu,\nu}$ is the associated generating function, where each map $M$ is counted with a weight $\alpha^{\eta(M)}$.
\end{remark}

\begin{conj}
Fix $n\geq 1$. For three partitions $\pi,\mu,\nu\vdash n$,
define the coefficient $f^\pi_{\mu,\nu}(\alpha)$ by
    the expansion
    $$\nabla(\bfp,\bfq) e_\pi(\bfp)=\sum_{\mu,\nu\vdash n}f^\pi_{\mu,\nu}m_\mu(\bfp) m_\nu(\bfq),$$
    where $e_\pi$ denotes the elementary symmetric function.
    Then $f^\pi_{\mu,\nu}(\alpha)\in\NN[\alpha].$
\end{conj}
This conjecture is a Jack analog of \cite[Conjecture 7.2]{BergeronHaglundIraciRomero2023}, it is however not clear whether one implies the other.

\subsection{On the Macdonald case}
Since the Macdonald super nabla operator introduced in \cite{BergeronHaglundIraciRomero2023} was the original motivation for studying its Jack analogue, it is natural to ask whether \cref{thm:commutation_relations} and \cref{thm:main_thm} admit counterparts in the Macdonald setting. We now briefly discuss these questions from two different points of view.

In \cite{BenDaliBonzomDolega2025}, a Macdonald analog of Chapuy--Do\l{}e\k{}ga operators was constructed; these operators are defined by nested commutators from the Macdonald operator, and are shown to have a combinatorial interpretation in terms of \emph{alternating paths}. However, computer experiments suggest that these operators do not satisfy commutation relations similar to the ones given in \cref{thm:commutation_relations}.

In a different direction, a Macdonald analog of Jack characters has been introduced in \cite{BenDaliDAdderio2023}. It was shown that these characters, as well as Macdonald polynomials, can be constructed using a new operator $\Gamma$. This can be thought of as parallel to the construction of Jack characters with Chapuy--Do\l{}e\k{}ga operators (see \cref{thm:BDD}). Inspired by this analogy, and using Tesler's identity, one can write\footnote{In order to avoid conflict of notation between the Jack and the Macdonald cases, we delay the statement of this result and its proof to Appendix \ref{sec:Macdonald}.} a formula for the Macdonald super nabla operator using a composition of classical operators; see Proposition \ref{prop:Macdonald}.
The main difficulty, however, lies in understanding the combinatorial structure of these operators, a question that falls outside the scope of this paper.

 This approach would in some sense be dual to the one used in this paper for the Jack case, where we started with the operators $\C_\ell$ which have a combinatorial structure, and the difficulty was to prove that they satisfy commutation relations, allowing to use them to ``construct'' the super nabla operator.

\subsection{Outline of the paper}
In \cref{sec:prel}, we give some definitions and preliminary results about symmetric functions and plethystic notation. We then discuss in \cref{sec shifted fcts} some properties of shifted symmetric functions and Jack characters. In \cref{sec:diff_operators}, we define the operators $\C_\ell$  and give some preliminary results about them.  \cref{sec:proof_main} is devoted to the proof of the main theorems. We establish in \cref{sec:connection-MJ} the connection between the super nabla operator and the Matching-Jack conjecture. Finally, we give in \cref{sec:Macdonald} a formula for the Macdonald super nabla operator.

\section{Preliminaries}\label{sec:prel}
\subsection{Partitions}\label{subsec Partitions}
A \textit{partition} $\lambda=[\lambda_1,...,\lambda_k]$ is a weakly decreasing sequence of positive integers $\lambda_1\geq...\geq\lambda_k>0$. We denote by $\YY$ the set of all integer partitions, including the empty partition. The integer $k$ is called the \textit{length} of $\lambda$ and is denoted $\ell(\lambda)$. The size of $\lambda$ is the integer $|\lambda|:=\lambda_1+\lambda_2+...+\lambda_k.$ If $n$ is the \textit{size} of $\lambda$, we say that $\lambda$ is a partition of $n$ and we write $\lambda\vdash n$. The integers $\lambda_1$,...,$\lambda_k$ are called the \textit{parts} of $\lambda$.

For $i\geq 1$, we denote $m_i(\lambda)$ the number of parts of size $i$ in $\lambda$. We then set 
\begin{equation*}
z_\lambda:=\prod_{i\geq1}m_i(\lambda)!i^{m_i(\lambda)}.  
\end{equation*}

\subsection{Symmetric functions and symmetric power-series}\label{ssec:SymFun}

We denote by $\mcS$ the algebra of symmetric functions on $\QQ(\alpha)$. For every partition $\lambda$ of length $k$, we denote respectively  by $m_\lambda$, $p_\lambda$, $h_\lambda$ the monomial, power-sum and complete homogeneous symmetric functions indexed by $\lambda$.

Since power-sum functions are a basis of the symmetric functions algebra, $\mathcal{S}$ can be identified with the polynomial algebra $\mcS\simeq\mathbb{Q}(\alpha)[ p_1,p_2,\dots].$
If $f$ is a symmetric function in the alphabet $\bfx$, it will be convenient when there is no ambiguity to denote with the same letter the function and the associated polynomial in the alphabet of power-sum functions $\bfp:=(p_1,p_2,..)$, i.e. $f(\bfx)\equiv f(\bfp)$. We consider the following scalar product on $\mcS$
\begin{equation}\label{eq scalar product}
  \langle p_\lambda,p_\mu\rangle=z_\lambda\alpha^{\ell(\lambda)}\delta_{\lambda,\mu},  
\end{equation}
\text{for any partitions $\lambda,\mu$}, where $\delta_{\lambda,\mu}$ denotes the Kronecker delta.
This scalar product is an $\alpha$-deformation of the \textit{Hall scalar product}, obtained by setting $\alpha=1$. 

If $\mcO$ is an operator on $\mcS$, we denote by $\mcO^\perp$ its adjoint with respect to the scalar product of \cref{eq scalar product}. In particular, for  $f\in\mcS$, $f^\perp$ denotes the adjoint of the multiplication by $f$. We then have for any $n\geq 1$
    $$p_n^\perp=\alpha\frac{n\partial}{\partial p_n}.$$

Finally, we denote $\mcSh$ the space of formal power-series in the variables $p_i$:
$$\mcSh:=\mcS_{\bfp}=\QQ(\alpha)\llbracket p_1,p_2,\dots\rrbracket,$$
We will also make use of the spaces of polynomials and formal series in two alphabets of variables $\bfp$ and $\bfq:=(q_1,q_2,\dots)$:
$$\mcS_{\bfp,\bfq}:=\QQ(\alpha)[p_1,p_2,\dots,q_1,q_2,\dots]\simeq \mcS\otimes\mcS$$
and 
$$\mcSh_{\bfp,\bfq}:=\QQ(\alpha)\llbracket p_1,p_2,\dots,q_1,q_2,\dots\rrbracket\simeq \mcSh\otimes\mcSh.$$
\subsection{Jack--Cauchy kernel}
Jack polynomials $\J_\lambda$ are orthogonal with respect to the scalar product of \cref{eq scalar product}:
$\langle \J_\lambda,\J_\theta\rangle=\delta_{\lambda,\theta}\ j^{(\alpha)}_\theta,$
where $j^{(\alpha)}_\theta\in \QQ(\alpha)$.
This orthogonality property, together with a triangularity property in the monomial basis, can be taken as a definition for Jack polynomials. We refer to \cite{Stanley1989} for more details.
The Jack--Cauchy kernel is the series in $\mcSh_{\bfp,\bfq}$ defined by
\begin{equation}\label{eq:Cauchy_kernel}
\Omega(\bfp,\bfq)=\sum_{\theta\in \YY}\frac{1}{j^{(\alpha)}_\theta}J^{(\alpha)}_\theta(\bfp)J^{(\alpha)}_\theta(\bfq)=\exp\left(\sum_{n\geq 1}\frac{p_nq_n}{\alpha n}\right)=\sum_{\mu\in\YY}\frac{p_\mu q_\mu}{z_\mu\alpha^{\ell(\mu)}}.
\end{equation}
We then have the following lemma.
\begin{lem}\label{lem:duality_Cauchy}
    For any operator $\mcO\in\End(\mcS)$, we have
    $$\mcO(\bfp)\cdot \Omega(\bfp,\bfq)=\mcO^\perp(\bfq)\cdot \Omega(\bfp,\bfq).$$
\end{lem}
\begin{proof}
    Define the coefficients $u_{\mu,\theta}$ by 
    $$\mcO(\bfp)\cdot \J_\mu(\bfp)=\sum_{\theta\in\YY}u_{\mu,\theta}\J_\theta(\bfp).$$
    We then have 
    $$u_{\mu,\theta}=\frac{1}{j^{(\alpha)}_\theta}\langle \mcO\J_\mu,\J_\theta\rangle=\frac{1}{j^{(\alpha)}_\theta}\langle \J_\mu,\mcO^\perp\J_\theta\rangle$$
which gives
$$\mcO^\perp(\bfq)\cdot \frac{\J_\theta(\bfq)}{j^{(\alpha)}_\theta}=\sum_{\mu\in\YY}u_{\mu,\theta}\frac{\J_\mu(\bfq)}{j^{(\alpha)}_\mu}.$$
We then get that
\begin{equation*}
  \mcO(\bfp)\cdot \Omega(\bfp,\bfq)=\sum_{\theta,\mu\in\YY}u_{\mu,\theta}\frac{\J_\theta(\bfp)\J_\mu(\bfq)}{j^{(\alpha)}_\mu}=\mcO^\perp(\bfq)\cdot \Omega(\bfp,\bfq).\qedhere  
\end{equation*}
\end{proof}

\subsection{Plethystic-type notation}\label{ssec:pleth_notation}
Let $u_1,\dots,u_N$ be a sequence of formal variables. We endow the ring $\QQ(\alpha)[u_1,\dots,u_N]$ with formal commutative operations $\oplus$ and $\ominus$. We call an element $\bA$ in this ring a \emph{signed alphabet}, it is of the form 
$$\bA=\mathop{\oplus}_{1\leq i\leq  k} x_i\mathop{\ominus}_{1\leq j\leq \ell} y_j,$$
where $x_i,y_j\in\QQ(\alpha)[u_i]$. The evaluation in $\bA$ is the algebra homomorphism $\mcS\rightarrow\QQ(\alpha)[u_i]$ defined on power-sums by
\begin{equation}\label{eq:subs_A}
  p_m(\bA)=\sum_{1\leq i\leq k}x_i^m-\sum_{1\leq j\leq \ell}y_j^m, \quad \text{for $m\geq 1$},  
\end{equation}
and $p_{\emptyset}[\bD_{\bfu}]=1$.
This notation is similar to the one used in \cite{AvalFerayNovelliThibon2015} for quasi-symmetric functions and is related to the Hopf algebra structure on symmetric functions; see e.g. \cite[Section 2]{GrinbergReiner2014}. Note also that the plethystic notation used here is different from the one used in the Macdonald theory (as in \cite[Chapter 1]{Haglund2008} for example).

The main example of signed alphabets in this paper is the following.
Consider a sequence of $N\geq 1$ variables $\bfu:=(u_1,u_2,\dots,u_N)$.

    We define $\bD_{\bfu}$ as the signed alphabet
    \begin{equation}\label{def:Du}
      \bD_{\bfu}:=\oplus_{1\leq i\leq N} (\alpha u_i-i+1)\ominus_{1\leq i\leq N}(\alpha u_i-i)\oplus (-N).  
    \end{equation}
Note that
$p_1[\bD_{\bfu}]=\sum_{1\leq i\leq N}(\alpha u_i-i+1)-\sum_{1\leq i\leq N}(\alpha u_i-i)-N=0.$

Similarly, to a partition $\lambda\in\YY$ we associate the alphabet
    $$\bD_{\lambda}:=\bD_{(\lambda_1,\dots,\lambda_N)}=\oplus_{1\leq i\leq N} (\alpha \lambda_i-i+1)\ominus_{1\leq i\leq N}(\alpha \lambda_i-i)\oplus (-N).$$
    One may notice that this definition does not depend on the choice of $N\geq \ell(\lambda)$. 
    The positive terms in this alphabet correspond to the \emph{$\alpha$-content} of the \emph{inner corners} of the Young diagram of $\lambda$, and the negative terms to the \emph{outer corners}; see \cite[Section 2]{DolegaFeray2016}.
    
\begin{example}
For $\lambda=[2,2]$, we have
$$\bD_{[2,2]}=(2\alpha)\ominus(2\alpha-1)\oplus(2\alpha-1)\ominus(2\alpha-2)\oplus(-2)
=(2\alpha)\ominus(2\alpha-2)\oplus(-2).$$
Let us compute $h_3\left[\bD_{[2,2]}\right]$. We start by expanding $h_3$ in the power-sum basis:
    $$h_3=p_{3}/3+p_{2}p_1/2+p_1^3/6.$$
    Since  $p_1[\bD_\lambda]=0$, we get
    $$h_3\left[\bD_{[2,2]}\right]=\frac{1}{3}p_3\left[\bD_{[2,2]}\right]=8 \alpha^2-8\alpha.$$
\end{example}

If $\bA$ is a signed alphabet, we define the \textit{plethystic exponential} 
\begin{equation}\label{eq:pleth_exp}
  \Exp\left[z\bA\right]=\sum_{n\geq 0}z^nh_n[\bA]=\exp\left(\sum_{k\geq 1}z^k\frac{p_k[\bA]}{k}\right).  
\end{equation}

With this notation in hand, the Cauchy formula for the alphabet $\bD_\lambda$ then reads
\begin{equation}\label{eq:Cauchy_D}
  \Exp\left[z\bD_\lambda\right]=\frac{1}{1+z\ell(\lambda)}\prod_{1\leq i\leq \ell(\lambda)}\frac{1-z (\alpha \lambda_i-i)}{1-z (\alpha \lambda_i-i+1)}.  
\end{equation}
This series is the Cauchy transform of the $\alpha$ deformed transition measure related to the partition $\lambda$, and its coefficients $h_n[\bD_\lambda]$ are the moments of this measure; see \cite[Section 2]{DolegaFeray2016}. The quantities $e_n[\bD_\lambda]$ and $p_n[\bD_\lambda]$ are usually referred to as boolean cumulants and fundamental functions of shape $\lambda$, respectively.

Note that the plethystic exponential satisfies the properties of a usual exponential
\begin{equation}\label{eq:pleth_exp_properties}
  \Exp\left[\ominus \bA\right]=\frac{1}{\Exp\left[ \bA\right]}, \qquad \text{and} \qquad\Exp\left[\bA\oplus \bB\right]=\Exp\left[ \bA\right]\Exp\left[ \bB\right].
\end{equation}

\section{Shifted symmetric functions and Jack characters}\label{sec shifted fcts}

In this section, we recall some useful properties about shifted symmetric functions and Jack characters.
\subsection{Shifted symmetric functions}
 We start by the following definition due to Lassalle.
\begin{defi}[\cite{Lassalle2008b}]
We say that a polynomial in $k$ variables
$(u_1,\dots,u_k)$ with coefficients in $\mathbb Q (\alpha)$ is
\textit{$\alpha$-shifted symmetric} if it is symmetric in the variables $u_i-i/\alpha$.
An $\alpha$-shifted symmetric function (or simply a shifted symmetric function)  is a sequence
$(f_k)_{k\geq1}$ of shifted symmetric polynomials of bounded degrees, such that for every $k\geq 1$, the function $f_k$ is
an $\alpha$-shifted symmetric polynomial in $k$ variables and 
\begin{equation}\label{eq shifted functions}
  f_{k+1}(u_1,\dots,u_k,0)=f_k(u_1,\dots,u_k).  
\end{equation}
We denote by $\mcSstar$ the algebra of shifted symmetric functions.
\end{defi}

Let $f$ be a shifted symmetric function and let $\lambda =
[\lambda_1,\dots,\lambda_k]$ be a partition. Then we denote
$$f(\lambda):=f(\lambda_1,\dots, \lambda_{k},0,\dots).$$ 

\begin{thm}[\cite{KnopSahi1996}]\label{thm:Knop_Sahi}
    Let $n\geq 0$, and let $g$ be a function on partitions. There exists a unique shifted symmetric function $f$ of degree less than or equal to $n$ such that $f(\lambda)=g(\lambda)$ for any $|\lambda|\leq n$.
\end{thm}

In particular, a shifted symmetric function $f$ is completely determined by its evaluation on partitions $(f(\lambda))_{\lambda\in\YY}$. 

\subsection{Jack characters}
Jack characters are an $\alpha$-deformation of the irreducible characters of the symmetric group introduced by Lassalle in \cite{Lassalle2008b}.
Given a partition $\mu$, the Jack character $\Jch_\mu$ is the function on partitions $\lambda$ defined by:
\begin{equation}\label{eq:def_Jch}
   \theta^{(\alpha)}_\mu(\lambda):=[p_\mu]\exp\left(\frac{\partial}{\partial p_1}\right)\J_\lambda(\bfp). 
\end{equation}
where  $[p_\mu] f(\bfp)$ denotes the coefficient of $p_\mu$ in $f$. In particular, $\Jch_\mu(\lambda)=0$ if $|\lambda|<|\mu|$.

Jack characters are closely related to the study of random Young diagrams under a Jack deformation of the Plancherel measure \cite{DolegaFeray2016,Sniady2019,CuencaDolegaMoll2023}. In \cite{BenDaliDolega2023}, a combinatorial interpretation for Jack characters has been given in terms of layered non-orientable maps, using a construction of the series of Jack characters with Chapuy--Do\l{}e\k{}ga operators (see \cref{thm:BDD}).

It turns out that Jack characters can be realized as shifted symmetric functions.
\begin{prop}[\cite{Lassalle2008b}]
    Fix a partition $\mu$. There exists a unique shifted symmetric function $f_\mu$ such that $f_\mu(\lambda)=\Jch_\mu(\lambda)$. 
\end{prop}
In the following we will denote by $\Jch_\mu$ the Jack character, both as a function on partitions and as a shifted symmetric function. Let $\mcSstar\llbracket \bfp\rrbracket$ denote the space of formal power-series in $\bfp$ whose coefficients are shifted symmetric in $\bfu:=(u_1,u_2,\dots)$.
We define the following series of Jack characters in $\mcSstar\llbracket \bfp\rrbracket$:
$$\mathcal{J}(\bfp;\bfu):=\sum_{\mu\in\YY}\Jch_\mu(\bfu) p_\mu.$$
We also write for $\lambda\in\YY$
$$\mathcal{J}(\bfp;\lambda):=\mathcal{J}(\bfp;\lambda_1,\lambda_2,\dots)\in \mcSh.$$

Notice that from \cref{eq:def_Jch},
\begin{equation}\label{eq:Jch}
          \mcJ(\bfp;\lambda)
        =\exp\left(\frac{\partial}{\partial p_1}\right)\cdot \J_\lambda(\bfp)
\end{equation}

In this paper, information on operators will be derived by studying their action on $\mcJ(\bfp;\bfu)$ thanks to the following lemma.
\begin{lem}\label{lem:action_Ju}
    Let $\mcO_1$ and $\mcO_2$ be two linear operators on $\mcSh$, satisfying 
    $$\mcO_1(\bfp)\cdot \mcJ(\bfp;\bfu)=\mcO_2(\bfp)\cdot \mcJ(\bfp;\bfu).$$
    Then $\mcO_1=\mcO_2$.
\end{lem}
\begin{proof}
    It is enough to prove that the two operators agree on the basis of $\mcSh$ given by Jack polynomials $\J_\lambda$. We proceed by induction on the size of $\lambda$. We assume that this is the case for all $\rho$ with size less than $n$. Fix $\lambda$ of size $n$. Taking $\bfu=\lambda$, we get
    $$\mcO_1\cdot \mcJ(\lambda)=\mcO_2\cdot \mcJ(\lambda).$$
    But we know that 
    \begin{align*}
        \mcJ(\bfp;\lambda)
        &=\exp\left(\frac{\partial}{\partial p_1}\right)\cdot \J_\lambda(\bfp)\\
        &=\J_\lambda(\bfp)+\sum_{|\rho|<|\lambda|}a_{\lambda,\rho}\J_\rho(\bfp),
    \end{align*} 
    for some coefficients $a_{\lambda,\rho}$. From the induction assumption we know that $\mcO_1\cdot \J_\rho=\mcO_2\cdot \J_\rho$ for any $|\rho|<|\lambda|$, and we get that $\mcO_1\cdot \J_\lambda=\mcO_2\cdot \J_\lambda$ as required.
\end{proof}

The following proposition gives the action of the dehomogenized nabla operator on the Jack character series.
\begin{prop}\label{prop:G_mcJ}
    We have
    $$\G(\bfp,\bfq)\cdot \mcJ(\bfp;\bfu)=\mcJ(\bfq;\bfu)\mcJ(\bfp;\bfu).$$
\end{prop}
\begin{proof}

    We know from \cref{thm:Knop_Sahi} that shifted symmetric functions are determined by their evaluations on partition. It is then enough to prove that 
    $$\G(\bfp,\bfq)\cdot \mcJ(\bfp;\lambda)=\mcJ(\bfq;\lambda)\mcJ(\bfp;\lambda).$$
    Using the definitions of $\G$ and \cref{eq:Jch}, we write
    \begin{align*}
        \G(\bfp,\bfq)\cdot \mcJ(\bfp;\lambda)
        &=\exp\left(\frac{\partial}{\partial p_1}\right)\exp\left(\frac{\partial}{\partial q_1}\right)
\nabla(\bfp,\bfq)
\exp\left(-\frac{\partial}{\partial p_1}\right)\exp\left(\frac{\partial}{\partial p_1}\right)\J_\lambda(\bfp)\\
&=\exp\left(\frac{\partial}{\partial p_1}\right)\exp\left(\frac{\partial}{\partial q_1}\right)
\J_\lambda(\bfp)\J_\lambda(\bfq)\\
&=\mcJ(\bfq;\lambda)\mcJ(\bfp;\lambda)
    \end{align*}
    as desired.
\end{proof}

\subsection{The shifting operator}
Fix a symmetric function $f$.
For any $N\geq 1$, define the polynomial $f^*_N$ by
\begin{equation}\label{eq:fstar_N}
  f^*_N(u_1,\dots,u_N):=f[\bD_{\bfu_N}],  
\end{equation}
where $\bfu_N=(u_1,\dots,u_N)$ and $\bD_{\bfu_N}$ is the signed alphabet defined in \cref{ssec:pleth_notation}. It follows from the definitions that $f^*_N$ is shifted symmetric in the variables $u_i$. Moreover, they satisfy the stability property
$f^*_{N+1}(u_1,\dots,u_{N},0)=f^*_{N+1}(u_1,\dots,u_{N}).$
This gives rise to a shifted symmetric function $f^*(u_1,u_2,\dots)$.

The functions $\{p_k^*\}_{k\geq 2}$ form an algebraic basis over $\QQ$ of $\mcSstar$; see \cite[Corollary 2.8]{IvanovOlshanski2002} and \cite[Section 2]{DolegaFeray2016}.   We then have the following.
\begin{prop}[\cite{IvanovOlshanski2002,DolegaFeray2016}]
    We have $\mcS^*\simeq \QQ(\alpha)[p_2^*,p_3^*,\dots]$.
\end{prop}
In particular, $\frac{\partial}{\partial p_k^*}$ is well defined as an operator on $\mcSstar$ for $k\geq 2$. We now define the operator on~$\mcS^*$
\begin{equation*}
    \Top_w=\Top_w(\bfu):=\exp\left(\sum_{n\geq 1}w^n\frac{(n+1)\partial}{\partial p^*_{n+1}}-\sum_{i\geq 1}p^*_{i}\frac{(i+1)
    \partial}{\partial p^*_{i+1}}\right).
\end{equation*}
We prove that this operator is a \emph{shifting operator}: it acts on a function $f$ by inserting a variable $w$ in the first position and shifting all the remaining variables one step to the right.
\begin{thm}\label{thm:Top_shift}
    Let $f$ be a shifted symmetric function. Then
    $$\Top_{\alpha w}\cdot f(u_1,u_2\dots)=f(w,u_1,u_2\dots).$$
\end{thm}
The operator $\Top_w$ can be thought of as a shifted symmetric version of the translation operator from the Macdonald theory (see \cite[Theorem 1.1]{BergeronGarsiaHaimanTesler1999}).

Before proving \cref{thm:Top_shift}, we introduce some notation. Set
$$a(w):=\sum_{n\geq 1}w^n\frac{(n+1)\partial}{\partial p^*_{n+1}}\quad\text{ and }\quad c:=-\sum_{i\geq 1}p^*_{i}\frac{(i+1)
    \partial}{\partial p^*_{i+1}}.$$
    We then have
    \begin{equation}\label{eq:comm_a_c_p}
      [a(w),p^*_{\ell}]=
    \begin{cases}
        \ell w^{\ell-1} &\text{if }\ell\geq 2\\
        0&\text{if }\ell=1.
    \end{cases}\quad \text{and}\quad [c,p^*_{\ell}]=
    \begin{cases}
        -\ell p^*_{\ell-1} &\text{if }\ell\geq 2\\
        0&\text{if }\ell=1.
    \end{cases}  
    \end{equation}
For any operator $\mcO$, we define $\ad_\mcO$ as the adjoint of $\mcO$: $\ad_\mcO(\mcO'):=[\mcO,\mcO'].$ We then have for $k\geq 1$
$$\ad_\mcO^k(\mcO')=\underbrace{\left[\mcO,\left[\mcO,\left[\dots[\mcO,\mcO']\dots\right]\right]\right]}_{\text{$k$ commutators}}$$

We have the following identity (see e.g. \cite{Wilcox1967})
\begin{equation}\label{eq:e^ad}
  \exp(\mcO)\mcO'\exp(-\mcO)=\sum_{k\geq0}\frac{\ad^k_\mcO(\mcO')}{k!}.  
\end{equation}
    
We start by the following lemma.
\begin{lem}\label{lem:shift}
For any $\ell\geq 1$, we have
   \begin{align*}
 p_\ell^*&(w,u_1,\dots):=(\alpha w)^\ell-(\alpha w-1)^\ell+\sum_{0\leq k\leq \ell}(-1)^{\ell-k}\binom{\ell}{k}p^*_{k}(u_1,\dots).
\end{align*}
\end{lem}
\begin{proof}
Fix $N\geq 1$. From the definitions (\cref{eq:subs_A,eq:fstar_N}),  we have
    \begin{align*}
 p_\ell^*&(w,u_1,\dots,u_N)-\left((\alpha w)^\ell-(\alpha w-1)^\ell\right)\\
 &=\sum_{1\leq i\leq N}\left((\alpha u_i-i)^\ell-(\alpha u_i-i-1)^\ell\right)+(-N-1)^\ell\\
 &=\sum_{0\leq k\leq \ell}(-1)^{\ell-k}\binom{\ell}{k}\left(\sum_{1\leq i\leq N}\left((\alpha u_i-i+1)^k-(\alpha u_i-i)^k\right)+(-N)^k\right)\\
 &=\sum_{0\leq k\leq \ell}(-1)^{\ell-k}\binom{\ell}{k}p^*_{k}(u_1,\dots,u_N).
\end{align*}
Finally, we take the limit $N\rightarrow\infty.$
\end{proof}

\begin{proof}[Proof of \cref{thm:Top_shift}]
We start by proving that, as operators on $\mcSstar$, we have
    \begin{equation}\label{eq:item2}
      \Top_{\alpha w} p^*_\ell(\bfu) \Top_{\alpha w}^{-1}=p_\ell^*(w,\bfu).  
    \end{equation}
We know from \cref{eq:e^ad} that
\begin{equation}\label{eq:conj_A}
  \Top_w p^*_\ell \Top_w^{-1}=\sum_{k\geq 0}\frac{\ad_{a(w)+c}^k(p^*_\ell)}{k!}=\sum_{k\geq 0}\frac{\left(\ad_{a(w)}+\ad_{c}\right)^k(p^*_\ell)}{k!},  
\end{equation}
Using \cref{eq:comm_a_c_p}, we get
\begin{align*}
  \Top_w p^*_\ell \Top_w^{-1}
  &=\sum_{k\geq 0}\frac{\ad_{c}^k(p^*_\ell)}{k!}+\sum_{k\geq 0}\frac{\ad_{a(w)}\left(\ad_{c}^k(p^*_\ell)\right)}{(k+1)!}\\
  &=\sum_{0\leq k\leq \ell-1}(-1)^{k}\binom{\ell}{k}p^*_{\ell-k}+
  \sum_{0\leq k\leq \ell-2}(-1)^{k}\binom{\ell}{k+1}w^{\ell-k-1}\\
&=\sum_{1\leq k\leq \ell}(-1)^{\ell-k}\binom{\ell}{k}p^*_{k}-\left((w-1)^{\ell}-w^{\ell}-(-1)^{\ell}\right)\\
&=\sum_{0\leq k\leq \ell}(-1)^{\ell-k}\binom{\ell}{k}p^*_{k}+\left(w^{\ell}-(w-1)^{\ell}\right).
\end{align*}
To obtain the last line, we use the fact that $p_\emptyset^*=1$.
    We now use \cref{lem:shift} to deduce that
    $$\Top_{\alpha w} p^*_\ell(\bfu) \Top_{\alpha w}^{-1}=p^*_\ell(w,\bfu).$$
For any partition $\mu$, we can write
$$\Top_{\alpha w}\cdot p^*_\mu(\bfu)=p^*_{\mu_1}(w,\bfu)\Top_{\alpha w}p_{\mu_2}(\bfu)\dots p^*_{\mu_{\ell(\mu)}}(\bfu),$$
and by induction we get
$$\Top_{\alpha w}\cdot p^*_\mu(\bfu)=p^*_{\mu}(w,\bfu)\Top_{\alpha w}\cdot 1=p_{\mu}(w,\bfu).$$
Since $p_\mu^*$ are a basis of $\mcSstar$, this finishes the proof of the theorem. 
\end{proof}

\section{Differential operators}\label{sec:diff_operators}
\subsection{Chapuy--Do\l{}e\k{}ga operators}\label{ssec:CD}
Following the notation of \cite{ChapuyDolega2022}, we define the space 
$$\mcS_Y:=\Span_{\QQ(\alpha)}\left\{y_i p_\lambda\right\}_{i\geq 0,\lambda \in\YY}=\mathop{\oplus}_{i\geq 0}y_i\mcS.$$
 We also consider the operators
\begin{equation*}
Y_+:=\sum_{i\geq 1}y_{i+1}\frac{\partial}{\partial y_i}: \mcS_Y\rightarrow\mcS_Y,
\end{equation*}
\begin{equation*}
\Lambda_Y:=\alpha\cdot\sum_{i,j\geq1}y_{i+j-1}\frac{j\partial^2}{\partial y_{i-1}\partial p_{j}}+\sum_{i,j\geq1} y_{i-1}p_j\frac{\partial}{\partial y_{i+j-1}} +(\alpha-1)\cdot \sum_{i\geq1}y_{i}\frac{i\partial}{\partial y_i}: \mcS_Y\rightarrow\mcS_Y,
\end{equation*}

$$\text{and }\quad \Theta_Y:=\frac{\partial}{\partial y_0}\Lambda_Y=\sum_{i\geq1}p_i\frac{\partial}{\partial y_i}: \mcS_Y\rightarrow\mcS.$$

For $\ell\geq 0$, the operator $\C_{\ell}(\bfp)$ is defined by 
\begin{equation}\label{eq:pos_Cl}
  \C_\ell(\bfp):=[u^\ell]\sum_{n\geq 1}\frac{(-1)^n}{n}\Theta_Y\left(Y_+\Lambda_Y+u Y_+\right)^{n}\frac{y_0}{\alpha}:\ \mcSh\rightarrow\mcSh.  
\end{equation}

\begin{remark}
    In \cite{BenDaliDolega2023,BenDali2025}, the operators $\C_\ell$ depend on an extra parameter $t$. This parameter can be recovered by taking the homogeneous parts of the operators. To simplify notation, we work here with the specialization $t=-1$. 
\end{remark}
Rather than using the explicit definition given above, we will mainly make use of  the fact that the operators $\{\C_\ell\}_{\ell\geq0}$ ``construct'' the Jack character series; see \cite[Eqs.~(28) and (65)]{BenDaliDolega2023}.
\begin{thm}[\cite{BenDaliDolega2023}]\label{thm:BDD}
    We have,
    \begin{equation}
        \mathcal{J}(\bfp;(w,\bfu))=\exp\left(\sum_{\ell\geq 0}(-\alpha w)^\ell\C_\ell(\bfp)\right)\mathcal{J}(\bfp;\bfu),
    \end{equation}
    where $(w,\bfu):=(w,u_1,u_2,\dots).$
\end{thm}

Combining \cref{thm:Top_shift} and \cref{thm:BDD}, we get the following.
\begin{cor}\label{cor:Top}
    For any $\ell\geq 0$, we have
    $$(-1)^\ell\C_\ell(\bfp)\mathcal{J}(\bfp;\bfu)=\begin{cases}
        \frac{(\ell+1)\partial}{\partial p^*_{\ell+1}(\bfu)}\cdot \mcJ(\bfp;\bfu)&\text{if }\ell\geq 1,\\
        -\sum_{i\geq 1}p_i^*(\bfu)\frac{(i+1)\partial}{\partial p^*_{i+1}(\bfu)}\cdot \mcJ(\bfp;\bfu)&\text{if $\ell=0$}.
    \end{cases}$$
\end{cor}
Notice that the operator on the left acts only on the alphabet $\bfp$ (the symmetric part) and on the right on the alphabet $\bfu$ (the shifted symmetric part). This simpler description of the action of the operators $(\C_\ell)_{\ell\geq 0}$ on the series $\mcJ(\bfp,u)$ once expressed in the variables $\bfu$, is the main motivation for developing the present approach based on shifted symmetric functions (see also \cref{prop:C_negative_mcJ} for a similar result for the operators $(\C_\ell)_{\ell< 0}$).
\begin{proof}
    Set $\mcA(\bfp;w):=\exp\left(\sum_{\ell\geq 0}(-\alpha w)^\ell\C_\ell(\bfp)\right)$.
    From \cref{thm:BDD} and \cref{thm:Top_shift}, we know that
    $\mcA(\bfp;w)\cdot \mathcal{J}(\bfp;\bfu)=\Top_{\alpha w}(\bfu)\cdot \mathcal{J}(\bfp;\bfu).$
    Since $\mcA$ and $\Top_w$ commute, we deduce that we also have
    $\log(\mcA(\bfp;w))\cdot \mathcal{J}(\bfp;\bfu)=\log(\Top_{\alpha w}(\bfu))\cdot \mathcal{J}(\bfp;\bfu).$ Finally, we extract the coefficient of $(\alpha w)^\ell$.
\end{proof}

\subsection{Nazarov--Sklyanin operators and the operators \texorpdfstring{$\C_\ell$}{Cl} for \texorpdfstring{$\ell<0$}{l<0}}\label{ssec:NS}

Nazarov--Sklyanin operators $\N_{\ell}$ are the operators on $\mcS$ defined for $\ell\geq 0$ by: 
$$\N_\ell(\bfp):=\frac{\partial}{\partial y_0} \Lambda_Y^{\ell} y_0.$$
As an example, one can check that $\N_0(\bfp):=1$, $\N_1(\bfp)=0$ and $\N_2(\bfp)=\sum_{i\geq 1}p_i\alpha\frac{i\partial}{\partial p_i}$ is the Euler operator.

The following theorem is due to Nazarov and Sklyanin \cite[Theorem 2]{NazarovSklyanin2013}. See also \cite{Moll2023,BenDaliDolega2023} for reformulation in terms of path operators.

\begin{thm}[\cite{NazarovSklyanin2013}]\label{thm:NS}
    For any $\ell\geq 0$ and any $\lambda\in\YY$, we have
    $$\N_\ell(\bfp)\cdot \J_\lambda(\bfp)=h_{\ell}[\bD_\lambda]\J_\lambda(\bfp).$$
\end{thm}
This theorem follows from \cite[Theorem~2]{NazarovSklyanin2013}. Since our notation and conventions differ from those in \cite{NazarovSklyanin2013}, we provide additional details in Appendix~\ref{App:NS} explaining how to translate between the two settings.

We introduce the following \textit{dehomogenized version} of Nazarov--Sklyanin operators:
\begin{equation}\label{eq:def_deh_NS}
    \tN_\ell(\bfp):=\exp\left(\frac{\partial}{\partial p_1}\right) \N_\ell(\bfp) \exp\left(-\frac{\partial}{\partial p_1}\right).
\end{equation}
Define the operator  $Y_-$on $\mcS_Y$ by
$$Y_{-}:=\sum_{i\geq 1}y_{i-1}\frac{\partial}{\partial y_{i}}.$$

The operator $\tN_\ell$ admits the following alternative expression.
\begin{prop}\label{prop:N_tilde}
    For any $\ell\geq 1$, we have 
    $$\tN_\ell(\bfp):=\frac{\partial}{\partial y_0}\left(Y_{-}+\Lambda_Y\right)^\ell y_0.$$
\end{prop}
\begin{proof}
First, notice the following commutation relation 
$$\left[\frac{\partial}{\partial p_1},\Lambda_Y\right]=Y_{-}.$$
Hence, we get that for any $n\geq 0$
    $$\frac{1}{n!}\ad_{\frac{\partial}{\partial p_1}}^n(\N_\ell(\bfp))=[v^n]\frac{\partial}{\partial y_0}\left(vY_-+\Lambda_Y\right)^\ell y_0.
    $$
    We then have from \cref{eq:e^ad} that
    $$\exp\left(\frac{\partial}{\partial p_1}\right) \N_\ell(\bfp) \exp\left(-\frac{\partial}{\partial p_1}\right)=e^{\ad_\frac{\partial}{\partial p_1}}\left(\N_\ell(\bfp)\right)$$
    which finishes the proof.
\end{proof}
One may notice that the operators $\tN_{\ell}$ are well defined both on $\mcS$ and on $\mcSh$. We now define the series
$$\tN(\bfp;v)=\sum_{\ell\geq 0}\tN_\ell(\bfp) v^{\ell}.$$
\begin{prop}\label{prop:N_mcJ}
    For any $\ell\geq 0$, we have
    \begin{equation}\label{eq:mcN_mcJ}
      \tN_\ell(\bfp)\cdot \mathcal{J}(\bfp;\bfu)=h_{\ell}^*(\bfu)\mathcal{J}(\bfp;\bfu).
    \end{equation}
    Equivalently,
    \begin{equation}\label{eq:mcNv_mcJ}
      \tN(\bfp;v)\cdot \mcJ(\bfp;\bfu)=\Exp[v\bD_{\bfu}]\cdot \mcJ(\bfp;\bfu).  
    \end{equation}
    
\end{prop}
\begin{proof}
Notice that both sides of \cref{eq:mcN_mcJ} are in $\mcSstar\llbracket \bfp\rrbracket.$
    By \cref{thm:Knop_Sahi}, it is enough to check the result when $\bfu$ is a partition, i.e. when  
    $u_i=\lambda_i$ for some $\lambda\in\YY$. Recall that in this case,  $\bD_{\bfu}=\bD_\lambda$.
    Using \cref{eq:Jch}, we have
    \begin{align*}
      \tN_\ell(\bfp)\cdot \mathcal{J}(\bfp;\lambda_1,\lambda_2,\dots)
      &= \exp\left(\frac{\partial}{\partial p_1}\right) \N_\ell(\bfp) \exp\left(-\frac{\partial}{\partial p_1}\right)\exp\left(\frac{\partial}{\partial p_1}\right) \J_\lambda(\bfp)\\
      &= \exp\left(\frac{\partial}{\partial p_1}\right) \N_\ell(\bfp)  \cdot \J_\lambda(\bfp).
    \end{align*}
    Using \cref{thm:NS}, we get 
    \begin{align*}
        \tN_\ell(\bfp)\cdot \mathcal{J}(\bfp;\lambda)
        &=\exp\left(\frac{\partial}{\partial p_1}\right)h_{\ell}[\bD_\lambda]  \cdot \J_\lambda(\bfp)\\
        &=h_{\ell}[\bD_\lambda] \mathcal{J}(\bfp;\lambda).
    \end{align*}
    
    This finishes the proof of \cref{eq:mcN_mcJ}. We multiply this equation by $v^{\ell}$ and we sum over $\ell\geq 0$ to obtain \cref{eq:mcNv_mcJ}.
\end{proof}

Recall that for $\ell>0$, $C_{-\ell}$  is the operator on $\mcSh$ defined by
\begin{equation}\label{eq:neg_Cl}
  \C_{-\ell}:=(-1)^\ell\left[v^{\ell+1}\right]\log\left(\tN(\bfp;v)\right),
\end{equation}
equivalently,
$$\tN(\bfp;v)=\exp\left(\sum_{\ell\geq 1} v^{\ell+1}(-1)^\ell\C_{-\ell}\right).$$

The following is an immediate consequence of \cref{prop:N_mcJ} and \cref{eq:pleth_exp}.

\begin{prop}\label{prop:C_negative_mcJ}
    For $\ell>0$, we have 
    $$\C_{-\ell}(\bfp)\cdot \mcJ(\bfp;\bfu)=\frac{(-1)^\ell }{\ell+1}p^*_{\ell+1}(\bfu)\mcJ(\bfp;\bfu).$$
\end{prop}

We deduce the following commutation relations.
\begin{thm}\label{thm:comm_rel_neg}
    For any $\ell,m<0$, we have
    $$\left[\C_\ell,\C_m\right]=0.$$
\end{thm}
\begin{proof}
    From \cref{prop:C_negative_mcJ}, one gets
    $$\left[\C_\ell(\bfp),\C_m(\bfp)\right]\cdot \mcJ(\bfp;\bfu)=0.$$
    We then conclude using \cref{lem:action_Ju}.
\end{proof}

We end this section by two useful properties about the operators $\C_\ell$.
\begin{lem}\label{lem:action_1}
    Given $\ell\leq 0$, we have $\C_{\ell}\cdot 1=0$.
\end{lem}
\begin{proof}
    This is a consequence of the facts that both $\Lambda_Y$ and $Y_-$ are identically 0 on $y_0.$
\end{proof}
Finally, the following proposition follows from the definitions of the operators $\C_\ell$.
\begin{prop}\label{prop:lowest_terms}
    We have
    $$\C_\ell=\begin{cases}
        (-1)^\ell\frac{p_\ell}{\ell\alpha}+\text{terms of higher degree}, &\text{if $\ell>0$},\\
        (-1)^\ell\frac{-\ell\alpha\partial}{\partial p_{-\ell}}+\text{terms of higher degree}, &\text{if $\ell<0$},\\
        -\sum_{i\geq 1}p_{i+1}\frac{i\partial}{\partial p_i}+\text{terms of higher degree},&\text{if $\ell=0$}.
    \end{cases}$$
\end{prop}
\begin{proof}
Define the degree of a monomial $y_ip_\mu\in \mcS_Y$ (in $y$ and $\bfp$) by $i+|\mu|$.
    Notice that, as operators on $\mcS_Y$,  $\Lambda_Y$ and $\Theta_Y$  preserve the degree, while $Y_+$ increases it by 1, and $Y_-$ decreases it by 1. It then follows that the operator $\Theta_Y\left(Y_+\Lambda_Y+u Y_+\right)^{n}\frac{y_0}{\alpha}$ increases the degree by $n$. As a consequence, for $\ell\geq0$, the lowest-degree term in $\C_\ell$ (defined in \cref{eq:pos_Cl}) is
    $$
    \begin{cases}
    \frac{(-1)^{\ell}}{\ell}\Theta_Y\left(Y_+\right)^{\ell}\frac{y_0}{\alpha}=\frac{(-1)^\ell p_\ell}{\alpha\ell} &\text{ if $\ell>0$},\\
    -\Theta_Y\left(Y_+\Lambda_Y\right)\frac{y_0}{\alpha}=-\sum_{i\geq 1}p_{i+1}\frac{i\partial}{\partial p_i}&\text{ if $\ell=0$.}
    \end{cases}
    $$
    Similarly, using \cref{prop:N_tilde}, the lowest-degree term in $\tN_\ell$ for $\ell\geq 2$ is 
    $$\frac{\partial}{\partial y_0} (Y_-)^{\ell-1}\Lambda_Y y_0=\frac{\alpha (\ell-1)\partial}{\partial p_{\ell-1}},$$
    recall also that $\tN_1=0$.
Combining this with \cref{eq:neg_Cl}, we conclude that for $\ell<0$ the lowest-degree term in $\C_\ell$ is $(-1)^\ell\frac{-\ell\alpha \partial}{\partial p_{-\ell}}$ as claimed.
\end{proof}

\section{Proof of the main results}\label{sec:proof_main}
\subsection{End of proof of \texorpdfstring{\cref{thm:commutation_relations}}{the first main theorem}}

\begin{thm}\label{thm:commutation_relations_pos_neg}
    For any $\ell,m>0$, we have
\begin{equation}\label{eq:commutation_relations}
      [\C_{-m},\C_\ell]=\delta_{\ell,m}.  
    \end{equation}
Moreover, 
\begin{equation}\label{eq:commutation_relations2}
    [\C_0,\C_{-m}]=
    \begin{cases}
        -m\C_{-(m-1)}&\text{if }m\geq 2\\
        0&\text{if }m=1.
    \end{cases}.
\end{equation}
\end{thm}
\begin{proof}
Let us prove \cref{eq:commutation_relations}. We apply the commutator on $\mcJ(\bfp;\bfu)$. 
Using \cref{prop:C_negative_mcJ} and \cref{cor:Top}, we get

\begin{align*}
  [\C_{-m}(\bfp),\C_\ell(\bfp)]\cdot \mcJ(\bfp;\bfu)
  &=\left(\C_{-m}(\bfp)\,(-1)^\ell\frac{(\ell+1)\partial}{\partial p^*_{\ell+1}} -\C_\ell(\bfp)\,\frac{(-1)^m}{m+1}p^*_{m+1}\right)\cdot \mcJ(\bfp;\bfu)\\
  &=\left(\frac{(-1)^m}{m+1}p^*_{m+1}\,\C_{-m}(\bfp) -\frac{(-1)^m}{m+1}p^*_{m+1}\,\C_\ell(\bfp)\right)\cdot \mcJ(\bfp;\bfu)\\
  &=\left[(-1)^\ell\frac{(\ell+1)\partial}{\partial p^*_{\ell+1}} ,\frac{(-1)^m}{m+1}p^*_{m+1}\right]\cdot \mcJ(\bfp,\bfu)\\
  &=\delta_{\ell,m}\mcJ(\bfp,\bfu).  
\end{align*}
Notice that we use here that an operator in $\bfp$ and an operator in $\bfu$ commute. From \cref{lem:action_Ju}, we get that $[\C_{-m},\C_\ell]=\delta_{\ell,m}.$ Similarly,
\begin{align*}
        [\C_0(\bfp),\C_{-m}(\bfp)]\cdot \mcJ(\bfp;\bfu)
        &=\left[\frac{(-1)^m}{m+1}p^*_{m+1}(\bfu),-\sum_{i\geq 1}p_i^*(\bfu)\frac{(i+1)\partial}{\partial p^*_{i+1}(\bfu)}\right]\cdot \mcJ(\bfp;\bfu)\\
        &=(-1)^mp^*_{m}(\bfu)\cdot \mcJ(\bfp;\bfu).
\end{align*}
Recall that $p^*_1(\bfu)=0$. We reapply \cref{prop:C_negative_mcJ} to write
\begin{equation*}
    [\C_0(\bfp),\C_{-m}(\bfp)]\cdot \mcJ(\bfp;\bfu)=
    \begin{cases}
        -m\C_{-(m-1)}\cdot \mcJ(\bfp;\bfu)&\text{if }m\geq 2\\
        0&\text{if }m=1.
    \end{cases}.
\end{equation*}
This finishes the proof of the theorem.
\end{proof}
Note that the key step in the proof is to re-express the action of the operators $\C_\ell(\bfp)$ on the series $\mcJ(\bfp,\bfu)$ as operators in the alphabet $\bfu$, in which they admit simpler expressions.

\subsection{Proof of Theorem \ref{thm:main_thm}}
Given a partition $\lambda=[\lambda_1,\dots,\lambda_\ell]$, define 
$$\C_\lambda:=\C_{\lambda_1}\dots\C_{\lambda_{\ell}}, \qquad \text{ and }\quad\C_{-\lambda}:=\C_{-\lambda_1}\dots\C_{-\lambda_{\ell}}.$$
\begin{lem}\label{lem:basis_C_lambda}
    Fix $f\in\mcSh$. Then there exists a unique family $\{u_\lambda\}$ in $\QQ(\alpha)$ such that 
    $$f(\bfp)=\sum_{\lambda\in\YY}u_\lambda \ \C_\lambda(\bfp)\cdot 1.$$
    Moreover, if $f$ is homogeneous with degree $n$ then $u_\lambda=0$ for $|\lambda|<n$.
\end{lem}
\begin{proof}
    The result follows from the triangularity property of the operators $\C_\lambda$ given in \cref{prop:lowest_terms}:
    \begin{equation*}
      \C_\lambda(\bfp)\cdot 1=(-1)^{|\lambda|}\frac{p_\lambda}{\alpha^{\ell(\lambda)}\prod_{i}\lambda_i}+\text{terms of higher degree}.\qedhere
    \end{equation*}
\end{proof}

\begin{lem}\label{lem:comm_C0_Clambda}
    Fix a partition $\lambda=[\lambda_1,\dots,\lambda_\ell].$ We have
    $$\C_0\C_\lambda\cdot 1=\sum_{1\leq i\leq \ell}(\lambda_i+1)\C_{\lambda_1}\dots\C_{\lambda_i+1}\dots\C_{\lambda_\ell}\cdot 1.$$
\end{lem}
\begin{proof}
    Since $\C_0\cdot 1=0$ (see \cref{lem:action_1}), then
    \begin{align*}
      \C_0\C_\lambda\cdot 1
      &=[\C_0,\C_\lambda]\cdot 1\\  
      &=\sum_{1\leq i\leq \ell}\C_{\lambda_1}\dots[\C_0,\C_{\lambda_i}]\dots\C_{\lambda_\ell}\cdot 1
    \end{align*}
    We conclude using \cref{eq:commutation_relations_C0}.
    \end{proof}

The following characterization for the operator $\G(\bfp,\bfq)$ has been proved in \cite{BenDali2025}.

\begin{thm}[{\cite[Theorem 1.6 and Proposition 6.2]{BenDali2025}}]\label{thm:characterization_G}
For any $\ell\geq 0$, we have \begin{equation}\label{eq com C-G k}
    \left(\C_{\ell}(\bfp)+\C_{\ell}(\bfq)\right)\cdot \G(\bfp,\bfq)=\G(\bfp,\bfq)\cdot \C_{\ell}(\bfp).
\end{equation}
Moreover, these equations characterize the operator $\G(\bfp,\bfq).$
\end{thm}

    Let $\tG(\bfp,\bfq)$ denote the right-hand side of \cref{thm:main_thm}:
$$\tG(\bfp,\bfq):=\exp\left(\sum_{\ell\geq 1}C_\ell(\bfq)C_{-\ell}(\bfp)\right)=\sum_{\lambda\in\YY}\frac{1}{\prod_{i\geq 1}m_i(\lambda)!}\C_\lambda(\bfq)\C_{-\lambda}(\bfp).$$
Before proving \cref{thm:main_thm}, let us check that $\tG(\bfp,\bfq)$ is well defined as an operator from $\mcS_{\bfp}$ to $\mcSh_{\bfp,\bfq}$. Recall that $\C_\lambda(\bfq)$ has only terms of degree at least $|\lambda|$ (see \cref{prop:lowest_terms}). Now, if we fix two partitions $\pi$ and $\mu$, and we consider the coefficient of $q_\mu$ in $\tG(\bfp,\bfq)\cdot p_\mu$, we know that only terms indexed by partitions $\lambda$ of size at most $|\mu|$ will contribute:
$$[q_\mu]\tG(\bfp,\bfq)\cdot p_\mu=[q_\mu]\sum_{|\lambda|\leq |\mu|}\frac{1}{\prod_{i\geq 1}m_i(\lambda)!}\C_\lambda(\bfq)\C_{-\lambda}(\bfp)\cdot p_\mu.$$
Since the sum is finite, we get an element in $\mcSh_{\bfp,\bfq}$. The same argument also shows that $\tG$ is well defined $\mcSh_{\bfp}$.

It is however not clear from the definitions that $\tG$ is well defined as an operator from $\mcS_{\bfp}$ to $\mcS_{\bfp,\bfq}$. This will be a consequence of \cref{thm:main_thm} which we now prove.
\begin{proof}[Proof of \cref{thm:main_thm}]
We will repeatedly use the fact that operators in the two different alphabets $\bfp$ and $\bfq$ trivially commute.
    In order to prove that $\tG=\G$, we need to check that $\tG$ satisfies the commutation relations of \cref{thm:characterization_G}.

We start by the case $\ell>0$.
We know from the commutation relations of \cref{eq:commutation_relations} that $\C_\ell(\bfp)$ commute with $\C_m(\bfq)\C_{-m}(\bfp)$ for any $m\neq \ell$. Moreover, we have
\begin{align*}
\exp\left(C_\ell(\bfq)C_{-\ell}(\bfp)\right)\C_\ell(\bfp)
&=\sum_{n\geq 0}\frac{1}{n!}
\C_\ell^n(\bfq)\C_{-\ell}^n(\bfp)\C_\ell(\bfp)\\
&=\C_\ell(\bfp)\sum_{n\geq 0}\frac{1}{n!}
\C_\ell^n(\bfq)\C_{-\ell}^n(\bfp)+
\sum_{n\geq 1}\frac{1}{(n-1)!}
\C_\ell^n(\bfq)\C_{-\ell}^{n-1}(\bfp)\\
&=\left(\C_\ell(\bfp)+\C_\ell(\bfq)\right)\exp\left(\C_\ell(\bfq)\C_{-\ell}(\bfp)\right),
\end{align*}
where the second line is obtained by applying \cref{eq:commutation_relations}. 
This gives the desired relations for $\ell>0$. We also get by a simple induction that for any partition $\lambda$
\begin{equation}\label{eq:Gt_comm_C_lambda}
\tG(\bfp,\bfq)\C_\lambda(\bfp)=\prod_{i\in\lambda}\left(\C_i(\bfp)+\C_i(\bfq)\right)\tG(\bfp,\bfq), 
\end{equation}
which will be useful later. We now focus on the case $\ell=0$. We want to prove that for any $f\in\mcS$, we have
$$\left(\C_{0}(\bfp)+\C_{0}(\bfq)\right)\cdot \tG(\bfp,\bfq)\cdot f(\bfp)=\tG(\bfp,\bfq)\cdot \C_{0}(\bfp)\cdot f(\bfp).$$
By \cref{lem:basis_C_lambda}, it is enough to prove that for any $\lambda$
\begin{equation}\label{eq:comm_tG_C0}
  \left(\C_{0}(\bfp)+\C_{0}(\bfq)\right)\cdot \tG(\bfp,\bfq)\cdot \C_\lambda(\bfp)\cdot 1=\tG(\bfp,\bfq)\cdot \C_{0}(\bfp)\cdot \C_\lambda(\bfp)\cdot 1.  
\end{equation}
Notice that \cref{lem:action_1} implies that $\tG\cdot 1=1$.
Using \cref{lem:comm_C0_Clambda} and \cref{eq:Gt_comm_C_lambda}, the left-hand side of \cref{eq:comm_tG_C0} gives
\begin{align*}
   \left(\C_{0}(\bfp)+\C_{0}(\bfq)\right)
&\prod_{i\in\lambda}\left(\C_i(\bfp)+\C_i(\bfq)\right)\tG(\bfp,\bfq)\cdot1
   \\
   &=\left(\C_{0}(\bfp)+\C_{0}(\bfq)\right)\cdot \prod_{1\leq i\leq \ell(\lambda)}\left(\C_{\lambda_i}(\bfp)+\C_{\lambda_i}(\bfq)\right)\cdot1\\
&=\sum_{1\leq i\leq \ell}(\lambda_i+1)\left(\C_{\lambda_1}(\bfp)+\C_{\lambda_1}(\bfq)\right)\dots
    \left(\C_{\lambda_i+1}(\bfp)+\C_{\lambda_i+1}(\bfq)\right)\dots\cdot 1
\end{align*}
Similarly, the right-hand side of \cref{eq:comm_tG_C0} gives
\begin{multline*}
  \sum_{1\leq i\leq \ell}(\lambda_i+1)\tG(\bfp,\bfq)\C_{\lambda_1}(\bfp)\dots\C_{\lambda_i+1}(\bfp)\dots\C_{\lambda_\ell}(\bfp)\cdot 1\\
    =\sum_{1\leq i\leq \ell}(\lambda_i+1)\left(\C_{\lambda_1}(\bfp)+\C_{\lambda_1}(\bfq)\right)\dots
    \left(\C_{\lambda_i+1}(\bfp)+\C_{\lambda_i+1}(\bfq)\right)\dots\left(\C_{\lambda_\ell}(\bfp)+\C_{\lambda_\ell}(\bfq)\right)\cdot 1.
\end{multline*}
This proves the commutation relation for $\ell=0$ and finishes the proof of the theorem.
\end{proof}

\section{Connection to the Matching-Jack conjecture}\label{sec:connection-MJ}
We consider the following power-series in three alphabets $\bfp$, $\bfq$ and $\bfr$
\begin{equation*}
      \tau^{(\alpha)}(\bfp,\bfq,\bfr):=\sum_{\theta\in \YY}\frac{1}{j^{(\alpha)}_\theta}J^{(\alpha)}_\theta(\bfp)J^{(\alpha)}_\theta(\bfq)J^{(\alpha)}_\theta(\bfr).
\end{equation*}

\begin{prop}\label{prop:tau_c}
   The power-sum expansion of $\tau^{(\alpha)}(\bfp,\bfq,\bfr)$ is given by
\begin{equation}\label{eq:tau_c}
  \tau^{(\alpha)}(\bfp,\bfq,\bfr)=\sum_{n\geq 0}\sum_{\pi,\mu,\nu\vdash n}\frac{c^\pi_{\mu,\nu}(\alpha)}{z_\pi\alpha^{\ell(\pi)}}p_\pi q_\mu r_\nu.  
\end{equation}
where $c^\pi_{\mu,\nu}(\alpha)$ are the coefficients defined in \cref{eq:def_c}.
\end{prop}
In Goulden--Jackson's original paper \cite{GouldenJackson1996}, \cref{eq:tau_c} is taken as the definition of the coefficients $c^\pi_{\mu,\nu}$.
\begin{proof}
    The function $\tau^{(\alpha)}$ is obtained from the Cauchy kernel by applying $\nabla$:
    \begin{equation}\label{eq:tau-Cauchy}
      \tau^{(\alpha)}(\bfp,\bfq,\bfr)=\nabla(\bfq,\bfr)\cdot \Omega(\bfp,\bfq).  
    \end{equation}
    Expanding the Cauchy kernel and using \cref{eq:def_c}, we get
    \begin{align*}
      \tau^{(\alpha)}(\bfp,\bfq,\bfr)
      &=\nabla(\bfq,\bfr)\sum_{n\geq 0}\sum_{\pi,\mu\vdash n}\frac{p_\pi q_\mu}{z_\pi\alpha^{\ell(\pi)}}\\
      &=\sum_{n\geq 0}\sum_{\pi,\mu,\nu\vdash n} \frac{c^\pi_{\mu,\nu}(\alpha)}{z_\pi\alpha^{\ell(\pi)}}p_\pi q_\mu r_\nu\qedhere
    \end{align*}
\end{proof}

Now define the series     
\begin{equation}\label{eq:def_G_series}
  G^{(\alpha)}(\bfp,\bfq,\bfr):=\exp\left(-\frac{p_1}{\alpha}\right)\exp\left(\frac{\partial}{\partial q_1}+\frac{\partial}{
  \partial r_1}\right)\tau^{(\alpha)}(\bfp,\bfq,\bfr).  
\end{equation}
It has been proved in \cite[Proposition 3.5]{BenDali2025} that this series encodes the structure coefficients of Jack characters. We prove that these also coincide with the coefficients $g^\pi_{\mu,\nu}$ defined in \cref{eq:def_g}.
\begin{prop}\label{prop:g}
    We have
\begin{equation}\label{eq:def_G}
  G^{(\alpha)}(\bfp,\bfq,\bfr):=\sum_{\pi,\mu,\nu\in\YY}\frac{g^\pi_{\mu,\nu}(\alpha)}{z_\pi \alpha^{\ell(\pi)}}p_\pi q_\mu r_\nu.  
\end{equation}

\end{prop}
\begin{proof}
    From  \cref{eq:tau-Cauchy,eq:def_G_series}, we get
\begin{align*}
  G^{(\alpha)}(\bfp,\bfq,\bfr)=\exp\left(\frac{\partial}{\partial q_1}+\frac{\partial}{
  \partial r_1}\right)\nabla(\bfq,\bfr)\exp\left(-\frac{p_1}{\alpha}\right)\cdot \Omega(\bfp,\bfq).  
\end{align*}
We now use \cref{lem:duality_Cauchy} and the fact that $(p_1/\alpha)^\perp=\frac{\partial}{\partial p_1}$ to get
\begin{align*}
  G^{(\alpha)}(\bfp,\bfq,\bfr)=\exp\left(\frac{\partial}{\partial q_1}+\frac{\partial}{
  \partial r_1}\right)\nabla(\bfq,\bfr)\exp\left(-\frac{\partial}{\partial q_1}\right)\cdot \Omega(\bfp,\bfq)=\G(\bfq,\bfr)\cdot \Omega(\bfp,\bfq).
\end{align*}
    We conclude using the expansion of $\Omega$ in the power-sum basis (\cref{eq:Cauchy_kernel}) and the definition of the coefficients $g^\pi_{\mu,\nu}$ (\cref{eq:def_g}).\end{proof}

\section*{Acknowledgements}
The author is very grateful to Valentin Féray for many useful suggestions on an earlier version of this paper. He also thanks Philippe Biane, Valentin Bonzom, Guillaume Chapuy, Michele D'Adderio, and Maciej Do\l{}e\k{}ga for several stimulating discussions about the theory of Jack and Macdonald polynomials. He is also grateful to the anonymous referee for valuable comments and suggestions. The author acknowledges support from the Center of Mathematical Sciences and Applications at Harvard University.

\appendix
\section{On the Macdonald case}\label{sec:Macdonald}
\begin{center}
\textit{In this appendix, all the operators and plethystic notation used are from the Macdonald theory and are different from the Jack notation used so far.}
\end{center}

We follow the notation of \cite{BergeronHaglundIraciRomero2023}: we will make use of the Macdonald plethystic notation, $\tH_\lambda[X]$ will denote modified Macdonald polynomials, $\nablaqt$ the nabla operator, and $\nablaqt_{\bfy}$ the super nabla operator defined by 
$$\supernabla\cdot \tH_\lambda[X]=\tH_\lambda[X]\tH_\lambda[Y].$$

We denote $\Op(\bfx)$ (respectively $\mcO(\bfy)$) the operator $\mcO$ acting on symmetric functions in the alphabet $\bfx$ (respectively $\bfy)$. We consider the translation operator $\Top_{Y}$ acting on symmetric functions in an alphabet $X$ by
\begin{equation}\label{eq:Topqt}
\Topqt_Y(\bfx)\cdot f[X]=\exp\left(\sum_{n\geq 1}\frac{p_n[Y]p^\perp_n[X]}{n}\right)\cdot f[X]= f[X+Y],
\end{equation}
and the operator $\mcP_Y$ defined by 
\begin{equation*}
    \mcP_Y(\bfx)\cdot f[X]=\Exp[XY]f[X],
\end{equation*}
with
$$\Exp[XY]:=\exp\left(\sum_{n\geq1}\frac{p_n[XY]}{n}\right)$$
We now define a dehomogenized operator $\dehsupernabla$
$$\dehsupernabla:=\Topqt_{1}(\bfx)\Topqt_1(\bfy)\supernabla \Topqt_{-1}(\bfx).$$
Note that, compared to the Jack case \cref{eq:def_G}, the operator $\exp\left(\frac{\partial}{\partial p_1}\right)$ is replaced by $\Top_1$.

We also recall that the (usual) \emph{nabla} operator $\nabla$ is defined by 
$$\nabla(\bfx)\cdot \tH_\lambda[X]=(-1)^{|\lambda|}\left(\prod_{i=1}^{\ell(\lambda)}\prod_{j=1}^{\lambda_i}q^{i-1}t^{j-1}\right)\tH_\lambda[X].$$

\begin{prop}\label{prop:Macdonald}
    We have
    \begin{equation}\label{eq:Macdonald}
      \dehsupernabla=\Pop_{\frac{1}{M}}(\bfx)\ \nablaqt(\bfx)^{-1}\Pop_{\frac{1}{M}}(\bfy)\ \nablaqt(\bfy)^{-1}\ \Topqt_Y(\bfx)\ \nablaqt(\bfx)\ \Pop_{\frac{-1}{M}}(\bfx),
    \end{equation}
    where $M:=(1-q)(1-t)$.
\end{prop}
\begin{proof}
We recall Tesler's identity (see~\cite{GarsiaHaimanTesler2001})
    \begin{equation} \label{eq Tesler}
\nabla \Pop_{-\frac{1}{M}}\Top_{1} \widetilde{H}_\lambda[X]=\Exp\left[-\frac{XD_\lambda}{M} \right].
\end{equation}
with $D_\lambda:=(1-t)\sum_{1\leq i\leq \ell(\lambda)}t^{i-1}(1-q^{\lambda_i})-1.$
It is easy to check that the two sides of \cref{eq:Macdonald} coincide on 
$\Top_{1}(\bfx)\cdot \tH_{\lambda}[X]$ for any $\lambda$. Moreover, using a triangularity argument one gets that $\Top_{1}(\bfx)\cdot \tH_{\lambda}[X]$ form a basis for the space of symmetric functions, which finishes the proof of the proposition.
\end{proof}

As in the Jack case, this formula is somehow related to a Macdonald version of Jack characters introduced in \cite{BenDaliDAdderio2023}, and to the operators $\Gamma$ allowing to construct them. However, the combinatorial structure of these operators is still not well understood.

Recall that Jack polynomials can be obtained from modified Macdonald polynomials via a plethystic transformation followed by a suitable limit. However, the decomposition given in \cref{prop:Macdonald} does not behave well under this limit: although the limit of the full composition of operators on the right-hand side is well defined, the limits of the individual operators are not. This issue also arises in several other Macdonald identities for which no Jack analogue is known—for instance, Tesler’s identity, from which \Cref{prop:Macdonald} is derived.

\section{About \texorpdfstring{\cref{thm:NS}}{}}\label{App:NS}
In this appendix, we explain how to derive \cref{thm:NS} from \cite[Theorem 2]{NazarovSklyanin2013}.
We will use bold letters for the notation of \cite{NazarovSklyanin2013}.

The operators $\mathbf{I}^{(k)}$ are defined in \cite[Section 6]{NazarovSklyanin2013} such that 
$$\mathbf{I}(z)\cdot \J_\lambda:=\sum_{k\geq 1}z^{k}\mathbf{I}^{(k)}\cdot \J_\lambda=z^{-1}\J_\lambda-(z^{-1}+\ell(\lambda))\prod_{1\leq i\leq \ell(\lambda)}\frac{z^{-1}+i-1-\alpha \lambda_i}{z^{-1}+i-\alpha\lambda_i}\J_\lambda \quad\text{for all $\lambda$.}$$
where both terms are expanded as series in $z$.
Using \cref{eq:Cauchy_D}, one gets that 
\begin{equation}\label{eq:action_I}
 \left(1-z\mathbf{I}(z)\right)^{-1} \cdot  \J_\lambda=\Exp[z\bD_\lambda]\cdot \J_\lambda,
\end{equation}

Moreover, from \cite[Theorem 2]{NazarovSklyanin2013}, we have that
\begin{equation*}
  \mathbf{I}^{(k)}=\sum_{i_1,\dots,i_k=1}^\infty p_{i_1}\mathbf{L}_{i_1,i_2}\dots\mathbf{L}_{i_{k-1},i_{k}}\frac{i_k\alpha\partial }{\partial p_{i_k}},  
\end{equation*}
where 
$$\mathbf{L}_{i,j}=j(\alpha-1)\delta_{i,j}+p_{i-j},$$
with the convention $p_k=\frac{-k\alpha\partial }{\partial p_{-k}}$ if $k<0$, and $p_0=0$.
We now claim that for any $\ell >0$, we have 
\begin{equation}\label{eq:N_I}
  \N_\ell=\sum_{r\ge1}\sum_{\substack{j_1,\dots,j_r>0\\j_1+\cdots +j_r=\ell}}\mathbf{I}^{(j_1)}\dots \mathbf{I}^{(j_r)}.  
\end{equation}
The key idea is to notice that the operators $\mathbf{L}_{i,j}$ and $\Lambda_Y$ have the same action, modulo the fact that $\Lambda_Y$ uses the extra alphabet $(y_i)_{i\geq 0}$ to keep track of the degree: increasing (respectively decreasing) the degree by $i$ corresponds to sending $y_k$ on $y_{k+i}$ (respectively on $y_{k-i}$). Forcing the index of $y$ to remain nonnegative corresponds to saying that for any $0\leq j\leq \ell$, all the terms in $\Lambda_Y^jy_0$ decrease (weakly) the degree in the alphabet $\bfp$. The situation is similar for the operator $\mathbf{I}^{(k)}$, except that the partial operators $\mathbf{L}_{i_j,i_{j+1}}\dots\mathbf{L}_{i_{k-1},i_{k}}\frac{i_k\alpha\partial }{\partial p_{i_k}}$ strictly decrease the degree (it has degree $-i_j<0$): we use here the fact that the operator $\mathbf{L}_{i,j}$ has degree $i-j$. Notice also that the operators $\mathbf{I}^{(k)}$ are homogeneous.

As a consequence, each term in the expansion of $\N_\ell$ corresponds to a unique sequence $(j_1,\dots,j_r)$ in \cref{eq:N_I} which corresponds 
to the indices for which the term is homogeneous (does not strictly decrease the degree in $\bfp$).

Now \cref{eq:N_I} gives
$$\sum_{\ell\geq 0}z^\ell\N_\ell=\left(1-z\mathbf{I}(z)\right)^{-1},$$
which combined with \cref{eq:action_I} gives \cref{thm:NS}.

\bibliographystyle{amsalpha}
\bibliography{biblio.bib}
   
\end{document}